\documentclass[12pt]{article}
\usepackage[latin9]{inputenc}
\usepackage{units}
\usepackage{amsmath}
\usepackage{amssymb}
\usepackage{amsthm}
\usepackage{color}
\usepackage{enumitem}

\makeatletter

\usepackage{graphicx} 

\usepackage{amsfonts}
\usepackage{algorithmic}\usepackage{algorithm}
\newtheorem{theo}{Theorem}[section]

\newtheorem{lem}[theo]{Lemma}
\newtheorem{remark}[theo]{Remark}

\newcommand{\bb}[1]{\boldsymbol{\mathrm{#1}}}

\newcommand{\RR}{\mathbb{R}}

\newcommand{\xx}{\bb{x}}

\newcommand{\Aa}{\bb{A}}
\newcommand{\Bb}{\bb{B}}
\newcommand{\Cc}{\bb{C}}
\newcommand{\Xx}{\bb{X}}
\newcommand{\Yy}{\bb{Y}}
\newcommand{\Vv}{\bb{V}}
\newcommand{\Uu}{\bb{U}}
\newcommand{\Mm}{\bb{M}}
\newcommand{\Dd}{\bb{D}}
\newcommand{\Ii}{\bb{I}}

\newcommand{\Ww}{\bb{W}}

\newcommand{\Llambda}{\bb{\Lambda}}

\makeatother

\begin{document}

\title{Almost-commuting matrices are almost jointly diagonalizable}

\author{Klaus Glashoff and Michael M. Bronstein\\
\small Institute of Computational Science\\
\small Faculty of Informatics\\
\small Universit{\`a} della Svizzera Italiana (USI)\\
\small Lugano, Switzerland\\
}
\maketitle
\begin{abstract}
We study the relation between approximate joint diagonalization of self-adjoint matrices and the norm of their commutator, and show that almost commuting self-adjoint matrices are almost jointly diagonalizable by a unitary matrix.
\end{abstract}

\section{Introduction}

The study of almost commuting matrices has been of interest in theoretical mathematics and physics communities, with the main question: are almost commuting matrices close to matrices that exactly commute? 
This question was answered positively for self-adjoint (Hermitian) matrices by Lin \cite{Huang_Lin}, and studied for additional different cases and settings \cite{Bernstein_Almost_commuting,Pearcy1979332,Rordam1996,Hastings2009,Filonov2010arXiv1008.4002F,Loring_Sorensen2010,Glebsky2010arXiv1002.3082G}.

In this paper, we study the relation between commutativity and joint diagonalizability of matrices: while it is well-known that commuting matrices are jointly diagonalizable, to the best of our knowledge, no results exist for almost-commuting matrices. 
Our result is that {\em almost commuting self-adjoint matrices are almost jointly  diagonalizable by a unitary matrix, and vice versa}, in a sense that will be explained later. 

Besides theoretical interest, this result has practical applications given the recent use of simultaneous approximate diagonalization of matrices in signal processing \cite{CardosoBlind1993,Cardoso_pertubation1994,Cardoso96jacobiangles}, machine learning \cite{2012arXiv1209.2295E}, and computer graphics \cite{KovBBKK:2013:EG}.  
In particular, Kovnatsky et al. \cite{KovBBKK:2013:EG} used joint diagonalizabiliy of Laplacian matrices as a criterion of similarity between 3D shapes (isometric shapes have jointly diagonalizable Laplacians). Since the joint diagonalization procedure is computationally expensive, the easily computable norm of the commutator can be used instead; our result justifies this use.

\section{Background} 

Let $\Aa, \Bb$ be two $n\times n$ complex  matrices. We denote by 
\begin{eqnarray*}
\| \Aa \|_\mathrm{F} &=& \textstyle \left( \sum_{ij}|a_{ij}|^2\right)^{1/2} = \left( \mathrm{tr}(\Aa^* \Aa) \right)^{1/2}; \\
\| \Aa \|_2 &=& \max_{\xx \in \RR^n : \| \xx\|_2 = 1}\| \Aa \xx \|_2 = \left( \lambda_{\mathrm{max}}(\Aa^*\Aa) \right)^{1/2}, 
\end{eqnarray*} 
the {\em Frobenius} and the {\em operator} norm (induced by the Euclidean vector norm) of $\Aa$, respectively. Here $\Aa^*$ is the adjoint (conjugate transpose) of $\Aa$.

We say that $\Aa, \Bb$ are {\em jointly diagonalizable} if there exists a unitary matrix $\Uu$ such that $\Uu^* \Aa \Uu = \Llambda_A$ and $\Uu^* \Bb \Uu = \Llambda_B$ are diagonal. 
In general, two matrices $\Aa, \Bb$ are not necessarily jointly diagonalizable, however, we can approximately diagonalize them by minimizing 
$$\min_{\Uu} J(\Aa,\Bb,\Uu)\,\,\, \mathrm{s.t.}\,\,\, \Uu^*\Uu = \Ii, $$
where 
$$J(\Aa,\Bb,\Uu) = \mathrm{off}(\Uu^* \Aa \Uu) + \mathrm{off}(\Uu^* \Bb \Uu),$$
and $\mathrm{off}(\Aa) = \sum_{i\neq j}|a_{ij}|^2$ is the sum of the squared absolute values of the off-diagonal elements. 
In the following, we denote $J(\Aa,\Bb) = \min_{\Uu^*\Uu = \Ii} J(\Aa,\Bb,\Uu)$.
Numerically, this optimization problem can be solved by a Jacobi-type iteration, referred to as the JADE algorithm \cite{Bunse-Gerstner:1993,Cardoso96jacobiangles}.

Furthermore, we say that $\Aa$ and $\Bb$ {\em commute} if $\Aa\Bb = \Bb\Aa$, and call $[\Aa,\Bb] = \Aa\Bb - \Bb\Aa$ their {\em commutator}. 
It is well-known that commuting self-adjoint matrices are jointly diagonalizable \cite{horn1990matrix}, which can be expressed as $\| [\Aa,\Bb]\|_\mathrm{F}=0$ iff $J(\Aa,\Bb) =0$. 
We are interested in extending this relation for the case $\| [\Aa,\Bb]\|_\mathrm{F}>0$ (respectively, $J(\Aa,\Bb) >0$). 

The main result of this paper is that if $\| [\Aa,\Bb] \|_\mathrm{F}$ is sufficiently small, then $J(\Aa,\Bb)$ is also small, and vice versa, i.e., almost commuting matrices are almost jointly diagonalizable.  
We can state this as the following

\begin{theo}[\bf main theorem] 
There exist functions $\epsilon_1(x), \epsilon_2(x)$ satisfying $\lim_{x\rightarrow 0} \epsilon_i (x) = 0$, $i=1,2$,  such that for any two self-adjoint $n\times n$ matrices $A,B$ with $\| A\|_\mathrm{F} = \| B\|_\mathrm{F} = 1$,
$$\epsilon_1( \| [A,B] \|_\mathrm{F} )  \leq J(A,B) \leq n\epsilon_2( \| [A,B] \|_\mathrm{F} ). $$ 
\end{theo}
The lower bound is discussed in Section~3. We show that this bound is independent of $n$ and is tight. 
The upper bound is discussed in Section~4. Besides showing the existence of the bounds, we also state them explicitly.

\section{Lower bound}

\begin{theo}[\bf lower bound] Let $\Aa, \Bb$ be self-adjoint matrices such that $\| \Aa\|_\mathrm{F} = \| \Bb\|_\mathrm{F} = 1$. Then, 
$$\frac{1}{4} \| [\Aa,\Bb] \|_\mathrm{F}^2  \leq J(\Aa,\Bb).$$  
\end{theo}

\begin{proof} Let us denote the minimizer   $\Vv = \mathop{\mathrm{argmin}}_{\Uu^*\Uu = \Ii} J(\Aa,\Bb,\Uu)$, and decompose 
\begin{eqnarray}
\label{eq:lower1}
\Vv^* \Aa \Vv &=& \Dd_A + \Xx; \\
\Vv^* \Bb \Vv &=& \Dd_B + \Yy,\nonumber
\end{eqnarray}
Here $\Dd_A, \Dd_B$ are diagonal matrices, and $\Xx=\Vv^* \Aa \Vv-\Dd_A, \Yy=\Vv^* \Bb \Vv-\Dd_B$ have zeroes on their diagonal.
This implies that $J(\Aa,\Bb) = \| \Xx \|_\mathrm{F}^2 + \| \Yy \|_\mathrm{F}^2$. 
Since $\| \Vv^* \Aa \Vv \|_\mathrm{F}^2 = \| \Dd_A\|_\mathrm{F}^2 + \| \Xx \|_\mathrm{F}^2$ and 
$\| \Vv^* \Bb \Vv \|_\mathrm{F}^2 = \| \Dd_B\|_\mathrm{F}^2 + \| \Yy \|_\mathrm{F}^2$, and using the invariance of the Frobenius norm to a unitary transformation, we get 
\begin{eqnarray}
\label{eq:lower2}
\| \Dd_A\|_\mathrm{F}^2 = \| \Vv^* \Aa \Vv \|_\mathrm{F}^2 - \| \Xx \|_\mathrm{F}^2 = \| \Aa \|_\mathrm{F}^2 - \| \Xx \|_\mathrm{F}^2 \leq \| \Aa \|_\mathrm{F}^2 \leq 1; 
\end{eqnarray}
in the same way, we establish that $ \| \Dd_B\|_\mathrm{F}^2 \leq 1$.

Rewriting~(\ref{eq:lower1}) as $\Aa = \Vv \Dd_A \Vv^* + \Vv  \Xx \Vv^*$ and 
$\Bb = \Vv \Dd_B \Vv^* + \Vv \Yy \Vv^*$, we get 
\begin{eqnarray*}
\Aa\Bb &=& \Vv \Dd_A \Vv^*  \Vv \Dd_B \Vv^*  +    \Vv \Dd_A \Vv^*  \Vv  \Yy \Vv^* + \Vv  \Xx \Vv^*  \Vv \Dd_B \Vv^* + \Vv  \Xx \Vv^* \Vv  \Yy \Vv^* \\
&=& \Vv \Dd_A  \Dd_B \Vv^*  +    \Vv \Dd_A  \Yy \Vv^* + \Vv  \Xx  \Dd_B \Vv^* + \Vv  \Xx  \Yy \Vv^*
\end{eqnarray*}
and 
\begin{eqnarray*}
\Aa\Bb &=& \Vv \Dd_B  \Dd_A \Vv^*  +    \Vv \Dd_B  \Xx \Vv^* + \Vv  \Yy  \Dd_A \Vv^* + \Vv  \Yy  \Xx \Vv^*. 
\end{eqnarray*}
Thus, we can express 
$$[\Aa,\Bb] = \Aa\Bb - \Bb\Aa = \Vv([\Dd_A,\Dd_B] + [\Dd_A,\Yy] + [\Xx,\Dd_B] + [\Xx,\Yy])\Vv^*;$$
since $\Dd_A, \Dd_B$ are diagonal, $[\Dd_A,\Dd_B]=0$, and we have 
$$[\Aa,\Bb] = \Vv([\Dd_A,\Yy] + [\Xx,\Dd_B] + [\Xx,\Yy])\Vv^* = \Vv([\Dd_A + \Xx,\Yy] + [\Xx,\Dd_B] )\Vv^*,$$
and finally, by the triangle inequality and the invariance of $\|\cdot\|$ with respect to unitary transformations
$$\|[\Aa,\Bb]\|_\mathrm{F} \leq \|[\Dd_A + \Xx,\Yy]\|_\mathrm{F} + \|[\Xx,\Dd_B]\|_\mathrm{F}.$$

Next, we use the bound of of B{\"o}ttcher and Wenzel\footnote{This bound was conjectured by B{\"o}ttcher and Wenzel \cite{Boettcher2005216}  for real square matrices, and proved later for different settings in \cite{vong2008proof,Bottcher20081864,lu2011normal,Lu20122531}. } $\| [\Aa,\Bb]\|_\mathrm{F}^2 \leq 2\|\Aa\|_\mathrm{F} \|\Bb\|_\mathrm{F}$  together with~(\ref{eq:lower1}) and (\ref{eq:lower2}) to get
\begin{eqnarray*}
\|[\Dd_A + \Xx,\Yy]\|_\mathrm{F} &\leq& \sqrt{2}\|\Dd_A + \Xx\|_\mathrm{F} \|\Yy\|_\mathrm{F}  
= \sqrt{2}\|\Aa\|_\mathrm{F} \|\Yy\|_\mathrm{F} \leq \sqrt{2} \|\Yy\|_\mathrm{F}; \\
\|[\Xx, \Dd_B]\|_\mathrm{F} &\leq& \sqrt{2}\|\Dd_B \|_\mathrm{F} \|\Xx\|_\mathrm{F}  \leq  
\sqrt{2}\|\Xx\|_\mathrm{F}. 
\end{eqnarray*}
This implies
\begin{eqnarray*}
\|[\Aa,\Bb]\|_\mathrm{F} &\leq& \|[\Dd_A + \Xx, \Yy]\|_\mathrm{F} + \|[\Xx,\Dd_B]\|_\mathrm{F}  
\leq \sqrt{2}( \|\Xx\|_\mathrm{F} + \|\Yy\|_\mathrm{F} ) \\
&\leq& \sqrt{2} (2\|\Xx\|^2_\mathrm{F} + 2\|\Yy\|^2_\mathrm{F})^{1/2} = 2 J^{1/2}(\Aa,\Bb), 
\end{eqnarray*}
which proves the theorem. 
\end{proof}

\begin{remark}
The bound is tight, which can be seen by considering the $2\times 2$ matrices 
\begin{center}
$\Aa_2=\left(\begin{array}{cc}
0.5 & 0.5\\
0.5 & -0.5
\end{array}\right), \,\,\,\,\,
\hat{\Bb_2}=\left(\begin{array}{cc}
-0.5 & -0.5+\epsilon\\
-0.5+\epsilon & 0.5
\end{array}\right),$
\end{center}
$\Bb_2:=\hat{\Bb}_2/\|\Bb_2\|$ for $\epsilon\rightarrow 0$.  
This example extends to any dimension $n>2$ by defining $n\times n$ matrices 
\begin{center}
$\Aa_n=\left(\begin{array}{cc}
\Aa_2 & 0\\
0 & 0
\end{array}\right),\,\,\,\,\,
\Bb_n=\left(\begin{array}{cc}
\Bb_2& 0\\
0 & 0
\end{array}\right).$

\end{center}

\end{remark}

\begin{figure}
\center{
  \includegraphics[width=1\linewidth]{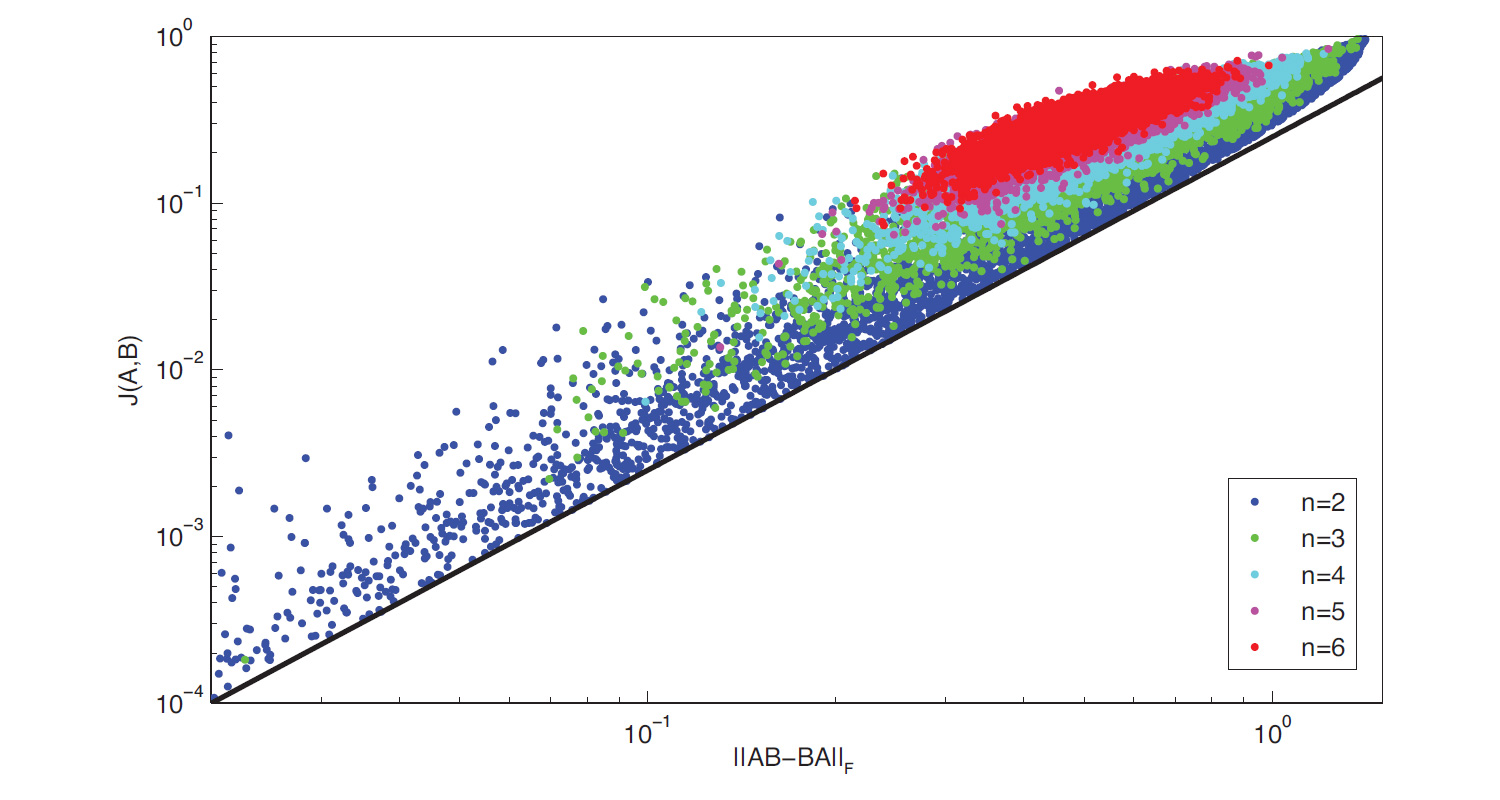}

  \caption{\label{fig:bounds} \small Visualization of the bounds for $100$ real symmetric $n\times n$ matrices drawn uniformly on the unit sphere for different values of $n$.  Lower bound is shown in black line. }
}
\end{figure}

\section{Upper bound}
\begin{theo}
There exists a function $\hat{\epsilon}(\delta)$ satisfying $\lim_{\delta \rightarrow 0}\hat{\epsilon}(\delta) = 0$ with the following property: If $A, B$ are two self-adjoint $n\times n$ matrices satisfying $\| A\|_2, \| B\|_2 \leq 1$, and $\| [A,B]\|_2 \leq \delta$, then 
$$J(\Aa,\Bb) \leq n\hat{ \epsilon}(\delta)$$. 
\end{theo}
In the proof of Theorem~4.1, we will use the following two auxiliary results. 
The first result is Huaxin Lin's theorem, asserting that almost commuting matrices are close to commuting matrices: 


\begin{theo}[\bf Lin 1995] 
There exists a function $\epsilon(\delta)$ satisfying $\lim_{\delta \rightarrow 0}\epsilon(\delta) = 0$ with the following property: If $A, B$ are two self-adjoint $n\times n$ matrices satisfying $\| A\|_2, \| B\|_2 \leq 1$, and $\| [A,B]\|_2 \leq \delta$, then there exists a pair $A', B'$ of \emph{commuting matrices}  satisfying $\| A - A' \|_2 \leq \epsilon(\delta)$ and $\| B - B' \|_2 \leq \epsilon(\delta)$.
\end{theo}

For a proof for the complex Hermitian case, we refer the reader to \cite{Huang_Lin,Rordam1996}. The first proof for the real case of symmetric matrices was given by Loring and S{\o}rensen \cite{Loring_Sorensen2010}.
The second result is the following property of the function $J$:

\begin{lem}
Let $\Aa, \Bb, \Cc, \Dd$ be self-adjoint $n\times n$ matrices, and let $\Uu$ denote a $n\times n$ unitary matrix. Then, 
$$|J(\Aa,\Bb,\Uu) - J(\Cc,\Dd,\Uu)| \leq  \|\Aa + \Cc\|_\mathrm{F} \|\Aa - \Cc\|_\mathrm{F} + 
\|\Bb + \Dd\|_\mathrm{F} \|\Bb - \Dd\|_\mathrm{F}.$$ 
\end{lem}

\begin{proof}

For notational convenience, let us define $J(\Aa,\Uu) = \mathrm{off}(\Uu^* \Aa \Uu)$, such that $J(\Aa,\Bb,\Uu) = J(\Aa,\Uu) + J(\Bb,\Uu)$. 
We can also express 
$$J(\Aa,\Uu) = \mathrm{tr}( (\Mm \circ (\Uu^* \Aa \Uu))^* (\Mm \circ (\Uu^* \Aa \Uu)) ),$$
where $\Mm$ is a matrix with elements $m_{ij} = 1 - \delta_{ij}$ and $\circ$ denotes the Hadamard (element-wise) matrix product. 
Using the relation $\mathrm{tr}((\Xx+\Yy)^*(\Xx-\Yy)) = \mathrm{tr}(\Xx^*\Xx - \Yy^*\Yy)$, we have
\begin{eqnarray*}
\begin{split}
|&J(\Aa,\Uu) - J(\Cc,\Uu)| =\\ 
&=|\mathrm{tr}( (\Mm \circ (\Uu^* \Aa \Uu))^* (\Mm \circ (\Uu^* \Aa \Uu)) ) - \mathrm{tr}( (\Mm \circ (\Uu^* \Cc \Uu))^* (\Mm \circ (\Uu^* \Cc \Uu)) ) | \\
&=|\mathrm{tr}( 
(\Mm \circ (\Uu^* \Aa \Uu) + \Mm \circ (\Uu^* \Cc \Uu) )^* 
(\Mm \circ (\Uu^* \Aa \Uu) - \Mm \circ (\Uu^* \Cc \Uu) )
) | \\
&=|\mathrm{tr}( 
(\Mm \circ (\Uu^* (\Aa+\Cc) \Uu ))^* 
(\Mm \circ (\Uu^* (\Aa-\Cc) \Uu ))
) |. 
\end{split}
%
\end{eqnarray*}
Employing the Cauchy-Schwartz inequality $|\mathrm{tr}(\Xx^*\Yy)| \leq \| \Xx\|_\mathrm{F}\| \Yy\|_\mathrm{F}$, we get 
\begin{eqnarray*}
|J(\Aa,\Uu) - J(\Cc,\Uu)| &=& |\mathrm{tr}( 
(\Mm \circ (\Uu^* (\Aa+\Cc) \Uu ))^* 
(\Mm \circ (\Uu^* (\Aa-\Cc) \Uu ))
) | \\
&\leq& 
\| \Mm \circ (\Uu^* (\Aa+\Cc) \Uu ) \|_\mathrm{F}\,
\| \Mm \circ (\Uu^* (\Aa-\Cc) \Uu ) \|_\mathrm{F}\\
&\leq& 
\| \Uu^* (\Aa+\Cc) \Uu \|_\mathrm{F}\,
\| \Uu^* (\Aa-\Cc) \Uu \|_\mathrm{F}\\
&=& 
\| \Aa+\Cc \|_\mathrm{F}\,
\| \Aa-\Cc \|_\mathrm{F}.
\end{eqnarray*}
By the same argument, 
$|J(\Bb,\Uu) - J(\Dd,\Uu)| \leq \| \Bb+\Dd \|_\mathrm{F}\,\| \Bb-\Dd \|_\mathrm{F}.$

\noindent Finally, 
\begin{eqnarray*}
|J(\Aa,\Bb,\Uu) - J(\Cc,\Dd,\Uu)| &=& 
|J(\Aa,\Uu) + J(\Cc,\Uu) - J(\Bb,\Uu) - J(\Dd,\Uu)|  \\
&\leq& 
|J(\Aa,\Uu) - J(\Cc,\Uu)| + |J(\Bb,\Uu) - J(\Dd,\Uu)| \\
&\leq& 
\| \Aa+\Cc \|_\mathrm{F}\,\| \Aa-\Cc \|_\mathrm{F} + 
\| \Bb+\Dd \|_\mathrm{F}\,\| \Bb-\Dd \|_\mathrm{F},
\end{eqnarray*}
which completes the proof of the lemma. 
\end{proof} 

We now state the proof of our upper bound: 
\begin{proof}[Proof of Theorem 4.1]

Let $\|[\Aa,\Bb]\|_\mathrm{F}\le \delta$ which implies $\|[\Aa,\Bb]\|_2\le \delta$, and $\|\Aa\|_2\le \|\Aa\|_\mathrm{F}\le 1, \|\Bb\|_2\le \|\Bb\|_\mathrm{F}\le 1$. By Lin's theorem, there are commuting matrices $\Aa',\Bb'$ such that 
$\| \Aa - \Aa' \|_\mathrm{F} \leq \sqrt{n}\| \Aa - \Aa' \|_2 \leq \sqrt{n}\epsilon(\delta)$.

Since $\Aa', \Bb'$ commute, they are jointly diagonalizable, implying that $J(\Aa',\Bb') = 0$, and that there exists a common diagonalizing matrix $\Ww' = \mathop{\mathrm{argmin}}_{\Uu^*\Uu = \Ii} J(\Aa',\Bb',\Uu)$. 
Applying Lemma 4.3, we get 
\begin{eqnarray*}
J(\Aa,\Bb) &=& J(\Aa,\Bb) - J(\Aa',\Bb') \leq J(\Aa,\Bb,\Ww') - J(\Aa',\Bb',\Ww')\\
&\leq& \|\Aa+\Aa'\|_\mathrm{F}\|\Aa-\Aa'\|_\mathrm{F} + 
\|\Bb+\Bb'\|_\mathrm{F}\|\Bb-\Bb'\|_\mathrm{F} \\
&\le&\|\Aa+\Aa+(\Aa'-\Aa)\|_\mathrm{F}\|\Aa-\Aa'\|_\mathrm{F} \\
& &+\|\Bb+\Bb+(\Bb'-\Bb)\|_\mathrm{F}\|\Bb-\Bb'\|_\mathrm{F} \\
&\leq& (2+\sqrt{n}\epsilon(\delta))\|\Aa-\Aa'\|_\mathrm{F} +(2+\sqrt{n}\epsilon(\delta))\|\Bb-\Bb'\|_\mathrm{F}\\
& \leq& 2(2+\sqrt{n}\epsilon(\delta))\sqrt{n}\epsilon(\delta).
\end{eqnarray*}
Now  $2(2+\sqrt{n}\epsilon(\delta))\sqrt{n}\epsilon(\delta)\le2 n(2/\sqrt{n}+\epsilon(\delta))\epsilon(\delta) \le 2 n(\sqrt{2}+\epsilon(\delta))\epsilon(\delta) =n\hat{\epsilon}(\delta)$ where we defined $\hat{\epsilon}(\delta)=2(\sqrt{2}+\epsilon(\delta))\epsilon(\delta)$, satisfying $\lim_{\delta \rightarrow 0}\hat{\epsilon}(\delta) = 0$ which finishes the proof of the theorem.

\end{proof}

\begin{remark}
The drawback of our Theorem~4.1 is  that it does not provide an explicit bound on $J(\Aa,\Bb)$ in terms of $\|\Aa,\Bb \|_\mathrm{F}$, but rather proves asymptotic behavior allowing to conclude that if two matrices almost commute, they are also almost jointly diagonalizable. 
In order to obtain an $\emph{explicit}$ bound, one can resort to different, more `constructive' alternatives to Lin's theorem:
%
\end{remark}

\begin{enumerate}[leftmargin=0cm,itemindent=.5cm,labelwidth=\itemindent,labelsep=0cm,align=left]

\item Hastings \cite{Hastings2009} showed that $\epsilon(\delta) = E(\delta^{-1})\delta^{1/5}$, where $E(x)$ is a function independent on $n$ that grows slower than any power of $x$,  
without, however, specifying the function $E$ explicitly.


\item There are different results \cite{Pearcy1979332,Glebsky2010arXiv1002.3082G,Filonov2010arXiv1008.4002F}, 
which, under the assumptions of Theorem 4.1, allow to calculate positive constants $c>0, \frac{1}{2} \le p,q\le 1$ such that if $\|[\Aa,\Bb]\|_F \leq \delta$, then 
\begin{equation}
\|\Aa-\Aa'\|_\mathrm{F}, \|\Bb-\Bb'\|_\mathrm{F}\le cn^{p}\delta^{q}.
\end{equation}
 By means of the arguments used for the proof of Theorem 4.1, together with the B\"{o}ttcher-Wenzel  bound $\delta\le \sqrt{2}$ \cite{Boettcher2005216} and the fact that $\frac{2}{n^p}\le \sqrt{2}$ for $\frac{1}{2}\le p\le 1,n\ge 2,$ this leads to the bound
\begin{eqnarray*}
J(\Aa,\Bb) &\leq&2n^{2p}(\nicefrac{2}{n^p}+c\delta^q)c\delta^q\\
&\le&2n^{2p}(\sqrt{2}+c\sqrt{2})c\delta^q\\
&\le& Cn^{2p} \|[\Aa,\Bb]\|_\mathrm{F}^{q}
\end{eqnarray*}
with $C=2\sqrt{2}(c+1)c $.
 For example,
Pearcy and Shields \cite{Pearcy1979332}\footnote{Pearcy and Shields use the operator norm $\|\cdot\|_2$ in the derivation of their bound, so the relation $\|\cdot\|_F\le \sqrt{n}\|\cdot\|_2$ has to be taken into account.  } obtained $c=\frac{1}{\sqrt{2}}, p = \frac{1}{4}, q = \frac{1}{2}$, 
 Glebsky \cite{Glebsky2010arXiv1002.3082G} $c=12, p = \frac{5}{12}, q = \frac{1}{6}$, and 
 Filonov and Kachkovskiy\footnote{In \cite{Filonov2010arXiv1008.4002F,Glebsky2010arXiv1002.3082G}  instead of the Frobenius norm the authors use the \emph{normalized} Frobenius norm $\|\cdot\|_{tr}=\frac{1}{\sqrt{n}}\|\cdot\|_\mathrm{F}$, so the assumptions and the assertion have to be adjusted accordingly.   }   \cite{Filonov2010arXiv1008.4002F} $c=2, p=\frac{3}{8}, q = \frac{1}{4}$. 

%
%
\end{enumerate}

\begin{remark}
We observed that none of the upper bounds derived by these theorems  lead to realistic values which are useful for numerical computations, so we do not discuss these results here in detail, and we leave this subject for further research.
\end{remark}

\section{Acknowledgement}
We thank David Wenzel and Terry Loring for pointing out some errors in a previous version.


\bibliographystyle{plain}\small
\bibliography{/Users/klausglashoff/Documents/OwnTexts/BIB/Spectral}

\begin{thebibliography}{10}

\bibitem{Bernstein_Almost_commuting}
A.~Bernstein.
\newblock Almost eigenvectors for almost commuting matrices.
\newblock {\em SIAM Journal on Applied Mathematics}, 21(2):232--235, 1971.

\bibitem{Boettcher2005216}
Albrecht B\"{o}ttcher and David Wenzel.
\newblock How big can the commutator of two matrices be and how big is it
  typically?
\newblock {\em Linear Algebra and its Applications}, 403(0):216 -- 228, 2005.

\bibitem{Bottcher20081864}
Albrecht B\"{o}ttcher and David Wenzel.
\newblock {The Frobenius norm and the commutator}.
\newblock {\em Linear Algebra and its Applications}, 429(8):1864--1885, 2008.

\bibitem{Bunse-Gerstner:1993}
Angelika Bunse-Gerstner, Ralph Byers, and Volker Mehrmann.
\newblock Numerical methods for simultaneous diagonalization.
\newblock {\em SIAM J. Matrix Anal. Appl.}, 14(4):927--949, October 1993.

\bibitem{Cardoso_pertubation1994}
J.~F. Cardoso.
\newblock Perturbation of joint diagonalizers.
\newblock Tech. Rep. 94D023, Signal Department, Telecom Paris, Paris, 1994.

\bibitem{Cardoso96jacobiangles}
Jean-Francois Cardoso and Antoine Souloumiac.
\newblock {Jacobi Angles For Simultaneous Diagonalization}.
\newblock {\em SIAM J. Mat. Anal. Appl}, 17:161--164, 1996.

\bibitem{CardosoBlind1993}
J.F. Cardoso and A.~Souloumiac.
\newblock {Blind beamforming for non-Gaussian signals}.
\newblock {\em Radar and Signal Processing, IEE Proceedings F}, 140(6):362
  --370, dec 1993.

\bibitem{2012arXiv1209.2295E}
D.~Eynard, K.~Glashoff, M.~M. Bronstein, and A.~M. Bronstein.
\newblock {Multimodal diffusion geometry by joint diagonalization of
  Laplacians}.
\newblock {\em ArXiv e-prints}, September 2012.

\bibitem{Filonov2010arXiv1008.4002F}
N.~{Filonov} and I.~{Kachkovskiy}.
\newblock {A Hilbert-Schmidt analog of Huaxin Lin's Theorem}.
\newblock {\em ArXiv e-prints}, August 2010.

\bibitem{Glebsky2010arXiv1002.3082G}
L.~{Glebsky}.
\newblock {Almost commuting matrices with respect to normalized Hilbert-Schmidt
  norm}.
\newblock {\em ArXiv e-prints}, February 2010.

\bibitem{Hastings2009}
M.B. Hastings.
\newblock Making almost commuting matrices commute.
\newblock {\em Communications in Mathematical Physics}, 291(2):321--345, 2009.

\bibitem{horn1990matrix}
R.~A. Horn and C.~R. Johnson.
\newblock {\em Matrix Analysis}.
\newblock Cambridge University press, 1990.

\bibitem{KovBBKK:2013:EG}
A.~Kovnatsky, M.~M. Bronstein, A.~M. Bronstein, K.~Glashoff, and R.~Kimmel.
\newblock Coupled quasi-harmonic bases.
\newblock {\em Computer Graphics Forum}, 2013.

\bibitem{Huang_Lin}
Huang Lin.
\newblock Almost commuting selfadjoint matrices and applications.
\newblock In {\em Fields Inst. Commun. Amer. Math. Soc.}, volume~13, pages
  193--233. Providence, RI, 1997.

\bibitem{Loring_Sorensen2010}
Terry~A. Loring and Adam P.~W. S{\o}rensen.
\newblock Almost commuting self-adjoint matrices - the real and self-dual
  cases.
\newblock {\em arXiv:1012.3494}, December 2010.

\bibitem{lu2011normal}
Z.~Lu.
\newblock {Normal Scalar Curvature Conjecture and its applications}.
\newblock {\em {Journal of Functional Analysis}}, 261:1284--1308, 2011.

\bibitem{Lu20122531}
Zhiqin Lu.
\newblock {Remarks on the B\"ottcher Wenzel inequality}.
\newblock {\em Linear Algebra and its Applications}, 436(7):2531 -- 2535, 2012.

\bibitem{Pearcy1979332}
Carl Pearcy and Allen Shields.
\newblock Almost commuting matrices.
\newblock {\em Journal of Functional Analysis}, 33(3):332 -- 338, 1979.

\bibitem{Rordam1996}
Mikael Rordam and Peter Friis.
\newblock {Almost commuting self-adjoint matrices - a short proof of Huaxin
  Lin's theorem.}
\newblock {\em Journal f\"{u}r die reine und angewandte Mathematik},
  479:121--132, 1996.

\bibitem{vong2008proof}
S-W. Vong and X-Q. Jin.
\newblock {Proof of B\''ottcher and Wenzel's conjecture}.
\newblock {\em Oper. Matrices}, 2:435--442, 2008.

\end{thebibliography}

\end{document}